\documentclass[a4paper,reqno,12pt]{amsart}
\usepackage{amsmath, amssymb, latexsym, amscd, amsthm,amsfonts,amstext}
\usepackage[mathscr]{eucal}
\usepackage{enumitem}

\makeatletter
\@namedef{subjclassname@2010}{%
  \textup{2010} Mathematics Subject Classification}
\makeatother
\usepackage{amsmath, amssymb, latexsym, amscd, amsthm,amsfonts,amstext}

\def\leq{\leqslant}
\def\geq{\geqslant}

\def\C{\mathbb{C} }

\def\leq{\leqslant}
\def\geq{\geqslant}

\theoremstyle{definition}

\def\0{\rm \bf 0}

\def\C{\mathbb{C} }

\def\geq{\geqslant}

\numberwithin{equation}{section}

\newtheorem{theorem}{Theorem}[section]
\newtheorem{remark}[theorem]{Remark}
\newtheorem{proposition}[theorem]{Proposition}

\begin{document}
\baselineskip=17pt

\centerline {\bf New local characterizations of the weighted energy class \(\mathcal{E}_{\chi,\mathrm{loc}}(\Omega)\)}
\bigskip

\centerline{\bf Hoang Nhat Quy}

\section{abstract}

		Let \(\Omega\subset\mathbb{C}^n\) be a hyperconvex domain and let \(\chi:\mathbb{R}^-\to\mathbb{R}^+\) be a decreasing function. This note studies the local weighted energy class \(\mathcal{E}_{\chi,\mathrm{loc}}(\Omega)\) introduced in \cite{HHQ13}.
		
		We establish two main results on local membership in this class. First, we prove a new local boundedness property for the weighted Monge--Amp\`ere energy: if \(u\in\mathrm{PSH}^-(\Omega)\) admits suitable local majorants in \(\mathcal{E}_{\chi,\mathrm{loc}}\) near the boundary of every relatively compact hyperconvex subdomain \(D\Subset\Omega\), then the weighted energy \(\int_K \chi(u)(dd^c u)^n\) remains locally finite for every compact set \(K\subset D\). This gives the first explicit local control of the energy functional and is new even in the unweighted setting. 
		
		Second, we obtain a substantial improvement concerning the local control of the Monge--Amp\`ere measure. We show that if, in addition to the boundary condition, \((dd^c u)^n\) is locally dominated by \((dd^c w)^n\) for some \(w\in\mathcal{E}_{\chi,\mathrm{loc}}(D)\) inside \(D\), then \(u\in\mathcal{E}_{\chi,\mathrm{loc}}(D)\). This domination condition is strictly weaker than the previous requirement of local finiteness of the weighted energy, thereby significantly enlarging the class of admissible functions.
		
		Our results extend and refine the local theory developed in \cite{Q24,Q25} and provide a more flexible framework for plurisubharmonic functions with possible singularities on compact subsets.
		
		\bigskip
		\noindent\textbf{Keywords:} plurisubharmonic functions, Cegrell classes, weighted energy classes, local classes, complex Monge--Amp\`ere operator, local energy boundedness, measure domination.

\section{Introduction}

Let \(\Omega\) be a hyperconvex domain in \(\mathbb{C}^n\). Denote by \(\mathrm{PSH}(\Omega)\) the set of plurisubharmonic functions on \(\Omega\) and by \(\mathrm{PSH}^-(\Omega)\) the subclass of negative plurisubharmonic functions. Bedford and Taylor, in their fundamental works \cite{BT1, BT82}, showed that the complex Monge--Amp\`ere operator \((dd^c \cdot)^n\) is well defined on the class of locally bounded plurisubharmonic functions. In \cite{Ce98}, Cegrell introduced the classes \(\mathcal{E}_0(\Omega)\), \(\mathcal{F}_p(\Omega)\) and \(\mathcal{E}_p(\Omega)\) on which the complex Monge--Amp\`ere operator is well defined. Later, in \cite{Ce04}, Cegrell introduced the classes \(\mathcal{F}(\Omega)\) and \(\mathcal{E}(\Omega)\) consisting of plurisubharmonic functions for which the operator \((dd^c \cdot)^n\) is well defined and continuous under decreasing sequences. He also proved that \(\mathcal{E}(\Omega)\) is the natural domain of definition of the Monge--Amp\`ere operator and is the largest class enjoying these properties. 

Furthermore, in \cite{BGZ}, Benelkourchi, Guedj and Zeriahi introduced the weighted energy class \(\mathcal{E}_\chi(\Omega)\), where \(\chi:\mathbb{R}^-\to\mathbb{R}^+\) is a decreasing function. Subsequently, in \cite{HHQ13}, the local weighted energy class \(\mathcal{E}_{\chi,\mathrm{loc}}(\Omega)\) was introduced.

Results concerning local properties of these classes have been studied extensively. From Definition 4.1 in \cite{Ce04}, the class \(\mathcal{E}(\Omega)\) is a local class. Recall that a class \(\mathcal{K}(\Omega)\subset\mathrm{PSH}^-(\Omega)\) is said to be \emph{local} if \(\varphi\in\mathcal{K}(\Omega)\) implies \(\varphi\in\mathcal{K}(D)\) for every hyperconvex domain \(D\Subset\Omega\), and if \(\varphi\in\mathrm{PSH}^-(\Omega)\) satisfies \(\varphi|_{\Omega_i}\in\mathcal{K}(\Omega_i)\) for all \(i\in I\) whenever \(\Omega=\bigcup_{i\in I}\Omega_i\), then \(\varphi\in\mathcal{K}(\Omega)\).

The proofs of Theorems 1.4 and 1.5 in \cite{Ah07} show that every function \(\varphi\in\mathcal{E}_\chi(\Omega)\) satisfies \(\lim_{z\to\xi}\varphi(z)=0\) for all \(\xi\in\partial\Omega\). Consequently, if \(\varphi\in\mathcal{E}_\chi(\Omega)\) and \(D\Subset\Omega\) is a relatively compact hyperconvex subdomain, then \(\varphi\notin\mathcal{E}_\chi(D)\). Thus \(\mathcal{E}_\chi(\Omega)\) is not a local class. The main result of \cite{HHQ13} establishes that \(\mathcal{E}_{\chi,\mathrm{loc}}(\Omega)\) is indeed a local class. In \cite{Q24}, the notions of local classes of type 1 and type 2 were introduced, and it was shown that \(\mathcal{E}_0(\Omega)\), \(\mathcal{F}(\Omega)\) and \(\mathcal{E}_\chi(\Omega)\) are local classes of type 2. Moreover, Theorem 3.8 in \cite{Q24} provides an explicit construction of \(\mathcal{E}_{\chi,\mathrm{loc}}(\Omega)\). In \cite{Q25}, some local characterizations of the existence of solutions to the complex Monge--Amp\`ere equation in the classes \(\mathcal{E}_\chi(\Omega)\) and \(\mathcal{E}_{\chi,\mathrm{loc}}(\Omega)\) were obtained (Theorem 3, 4).

The main purpose of this note is to establish several local conditions on a function \(u\in\mathrm{PSH}^-(\Omega)\) ensuring that \(u\) belongs to the class \(\mathcal{E}_{\chi,\mathrm{loc}}(\Omega)\).

\section{Preliminaries}

\noindent{\bf 3.1 Local classes.} 
We recall the definition of local classes from \cite{Q24}. Let \(\mathcal{K}\) be a family of functions on open sets in a topological space \(X\). We say that
\begin{enumerate}[label=\textbf{\roman*)}]
	\item \(\mathcal{K}\) is a \emph{local class of type 1} if
	\[
	\mathcal{K}(U)|_V = \{f|_V : f\in\mathcal{K}(U)\}\subset \mathcal{K}(V)
	\]
	for all open sets \(V\subset U\subset X\).
	\item \(\mathcal{K}\) is a \emph{local class of type 2} if, for any function \(f\) on \(U\) such that for every \(x\in U\) there exists a neighbourhood \(V\) of \(x\) with \(f|_V\in\mathcal{K}(V)\), then \(f\in\mathcal{K}(U)\).
	\item \(\mathcal{K}\) is a \emph{local class} if it is both a local class of type 1 and of type 2.
	\item We denote by \(\mathcal{K}^1_{\mathrm{loc}}(X)\) the family of functions on \(X\) such that for every \(x\in X\) there exists a neighbourhood \(U\) of \(x\) with \(f|_U\in\mathcal{K}(X)|_U\).
	\item We denote by \(\mathcal{K}^2_{\mathrm{loc}}(X)\) the family of functions on \(X\) such that for every \(x\in X\) there exists a neighbourhood \(U\) of \(x\) with \(f|_U\in\mathcal{K}(U)\).
\end{enumerate}

Proposition 2.3 in \cite{Q24} shows the following relations between these classes:
\begin{enumerate}[label=\textbf{\roman*)}]
	\item \(\mathcal{K}\subset\mathcal{K}^1_{\mathrm{loc}}\).
	\item If \(\mathcal{K}\) is a local class of type 1, then \(\mathcal{K}^1_{\mathrm{loc}}\subset\mathcal{K}^2_{\mathrm{loc}}\).
	\item \(\mathcal{K}\) is a local class of type 2 if and only if \(\mathcal{K}^2_{\mathrm{loc}}\subset\mathcal{K}\).
	\item \(\mathcal{K}\) is a local class if and only if \(\mathcal{K}^1_{\mathrm{loc}}=\mathcal{K}^2_{\mathrm{loc}}=\mathcal{K}\).
\end{enumerate}

\noindent{\bf 3.2 Cegrell's classes.} 
We recall some classical pluricomplex energy classes from \cite{Ce98} and \cite{Ce04}. Let \(\Omega\) be a bounded hyperconvex domain in \(\mathbb{C}^n\). We define
$$
\begin{aligned}
	&\mathcal E_0(\Omega)=\{\varphi\in\text{L}^\infty(\Omega): \lim_{z\to\xi}\varphi(z)=0,\;\forall\xi\in\partial\Omega \text{ and } \int_\Omega(dd^c\varphi)^n<+\infty\},\\
	&\mathcal F(\Omega)=\left\{\varphi\in\text{PSH}^-(\Omega): \exists\mathcal E_0(\Omega)\ni\varphi_j\searrow\varphi, \sup_{j\ge 1}\int_\Omega(dd^c\varphi_j)^n<+\infty\right\},\\
\end{aligned}$$
and
$$\begin{aligned}
		\mathcal E(\Omega)=\left\{\varphi\in\text{PSH}^-(\Omega):\right.& \forall  z_0\in\Omega,\exists \text{ a neighbourhood } \omega\ni z_0,\\
		& \left. \mathcal E_0(\Omega)\ni\varphi_j\searrow\varphi\text{ on } \omega, \sup_{j\ge 1}\int_\Omega (dd^c\varphi_j)^n<+\infty\right\}.
\end{aligned}
$$

From \cite{Q24} we have the following results:
\begin{enumerate}[label=\textbf{\roman*)}]
	\item \(\mathcal{E}_{0,\mathrm{loc}}^2(\Omega)=\{0\}\) and \(\mathcal{E}_{0,\mathrm{loc}}^1(\Omega)=\mathrm{PSH}^-\cap L^\infty(\Omega)\).
	\item \(\mathcal{F}_{\mathrm{loc}}^2(\Omega)=\{0\}\) and \(\mathcal{F}_{\mathrm{loc}}^1(\Omega)=\mathcal{E}(\Omega)\).
\end{enumerate}

\noindent{\bf 3.3 Weighted energy classes.} 
We recall the weighted pluricomplex energy classes introduced in \cite{BGZ} and \cite{HHQ13}:
\[
\mathcal{E}_\chi(\Omega)=\Bigl\{\varphi\in\mathrm{PSH}^-(\Omega):\ \exists\,\varphi_j\in\mathcal{E}_0(\Omega),\ \varphi_j\downarrow\varphi,\ 
\sup_{j\geq 1}\int_\Omega\chi(\varphi_j)(dd^c\varphi_j)^n<+\infty\Bigr\},
\]
where \(\chi:\mathbb{R}^-\to\mathbb{R}^+\) is a decreasing function, and
\[
\mathcal{E}_{\chi,\mathrm{loc}}(\Omega)=\Bigl\{\varphi\in\mathrm{PSH}^-(\Omega):\ \exists\,\psi_D\in\mathcal{E}_\chi(\Omega)\text{ such that }\varphi=\psi_D\text{ on }D,\ \forall\,D\Subset\Omega\Bigr\}.
\]

If we assume further that \(\chi(2t)\leq a\chi(t)\) for some constant \(a>1\), then \(\mathcal{E}_\chi(\Omega)\) is a cone, i.e.,
\begin{enumerate}[label=\textbf{\roman*)}]
	\item \(\varphi\in\mathcal{E}_\chi(\Omega)\), \(\psi\in\mathrm{PSH}^-(\Omega)\), \(\psi\geq\varphi\) \(\implies\) \(\psi\in\mathcal{E}_\chi(\Omega)\);
	\item \(\varphi,\psi\in\mathcal{E}_\chi(\Omega)\) \(\implies\) \(k\varphi + l\psi\in\mathcal{E}_\chi(\Omega)\) for all \(k,l\geq 0\).
\end{enumerate}

Theorem 3.8 in \cite{Q24} shows that \((\mathcal{E}_\chi)^1_{\mathrm{loc}}(\Omega)=\mathcal{E}_{\chi,\mathrm{loc}}(\Omega)\) and \((\mathcal{E}_\chi)^2_{\mathrm{loc}}(\Omega)=\{\varphi\in\mathrm{PSH}^-(\Omega):\ \varphi\geq -t_0\}\), where \(t_0=\sup\{t\in\mathbb{R}^-:\chi(t)>0\}\).

\section{A characterization of the class $\mathcal E_{\chi,loc}$}

The following proposition establishes a useful invariance property of the class \(\mathcal{F}(\Omega)\) under local modifications outside a compact set. This property will serve as a key tool in the proofs of subsequent results.
\begin{proposition}\label{pro21}
	Let \(\Omega\) be a hyperconvex domain and \(K \Subset \Omega\). Assume that \(u, v \in \mathrm{PSH}^-(\Omega)\) and \(u = v\) on \(\Omega \setminus K\). Then \(u \in \mathcal{F}(\Omega)\) if and only if \(v \in \mathcal{F}(\Omega)\).
\end{proposition}

\begin{proof}
	Assume \(u \in \mathcal{F}(\Omega)\). Then we are going to prove that $v\in\mathcal F(\Omega)$. Take the sequence \((\varphi_j)_{j \ge 1} \subset \mathcal{E}_0(\Omega)\) such that
	\[
	\varphi_j \downarrow u \quad \text{on } \Omega \qquad \text{and} \qquad \sup_{j \ge 1} \int_{\Omega} (dd^c \varphi_j)^n < +\infty.
	\]	
	For each \(j \ge 1\), we set
	\[
	u_j = \max(u, j \varphi_j), \qquad v_j = \max(v, j \varphi_j).
	\]	
	Then $u_j, v_j \in \mathcal{E}_0(\Omega) \quad \text{for all } j \ge 1$ and $u_j \downarrow u, \quad v_j \downarrow v \quad \text{on } \Omega$. \\	
	Since \(u = v\) on \(\Omega \setminus K\), it follows that
	\[
	u_j = v_j \quad \text{on } \Omega \setminus K \quad \text{for every } j \ge 1.
	\]	
	By Stokes' theorem we obtain that
	\[
	\int_{\Omega} (dd^c u_j)^n = \int_{\Omega} (dd^c v_j)^n \quad \text{for every } j \ge 1.
	\]	
	Therefore
	\[
	\sup_{j \ge 1} \int_{\Omega} (dd^c u_j)^n = \sup_{j \ge 1} \int_{\Omega} (dd^c v_j)^n.
	\]	
	Since $u\in\mathcal F(\Omega)$, by Lemma 2.1 in \cite{CeZe03} we have
	
	\[\sup_{j \ge 1} \int_{\Omega} (dd^c v_j)^n<+\infty.\]		
	These show that \(v \in \mathcal{F}(\Omega)\).
	
	The converse implication follows by interchanging the roles of \(u\) and \(v\).
\end{proof}

As an immediate application of the above invariance, we recall a classical local characterization for the unweighted class \(\mathcal{E}(\Omega)\).

\begin{theorem}\label{the22}
	Let $\Omega, D$ be bounded domains in $\mathbb C^n$ and $D\Subset \Omega$. Let $u\in PSH^-(\Omega)$ such that for all $z\in\partial D$, there exist a ball $\mathbb B(z,r)$ and $v\in\mathcal E(\mathbb B(z,r))$ with $u\geq v$ on $\partial D\cap\mathbb B(z,r)$. Then $u\in\mathcal E(D)$. 
\end{theorem}
\begin{proof}
	Take $\mathbb B(0,R_0)\Supset\Omega$. For $z\in\partial D$ there exist $v_z\in\mathcal E(\mathbb B(z,r_z))$ such that
	$$u\geq v_z\quad\text{ on } \partial D\cap\mathbb B(z,r_z).$$
	Take $w_z\in\mathcal F(\mathbb B(z,r_z))$ such that
	$$w_z=v_z\quad\text{ on } \mathbb B(z,\frac{r_z}{2}).$$
	By the subextension theorem (\cite{CeZe03}), there exist $\varphi_z\in\mathcal F(\mathbb B(0,R_0)$ such that $\varphi_z\leq w_z$ on $\mathbb B(z,\frac{r_z}{2})$. \\
	Since $\partial D$ compact, we can choose $z_1, z_2, ..., z_h\in\partial D$ such that 
	$$\partial D\subset\cup_{j=1}^h\mathbb B(z_j,\frac{r_{z_j}}{2}).$$
	Set $\varphi=\varphi_{z_1}+\varphi_{z_2}+...+\varphi_{z_h}$. Then $\varphi\in\mathcal F(\mathbb B(0,R_0))$ and $u\geq \varphi$ on $\partial D$. \\
	Set
	$$
	\psi=
	\begin{cases}
		u\quad&\text{ on } D\\
		\max(u,\varphi)\quad&\text{ on } \Omega\setminus D.
	\end{cases}
	$$
	Then $\psi\in PSH^-(\mathbb B(0,R_0))$ and $\psi\geq \varphi$ on $\Omega\setminus D$. By Proposition \ref{pro21} we have $\psi\in\mathcal F(\mathbb B(0, R_0))$. Since $\mathcal E(\Omega)$ has local property, we get $u\in\mathcal E(D)$. 
\end{proof}

We now turn to the weighted setting. In order to obtain local membership criteria for \(\mathcal{E}_{\chi,\mathrm{loc}}\), we first need a subextension theorem that preserves both the Monge--Ampère mass and the weighted energy.

\begin{proposition}\label{pro43}
	Let \(D \Subset \Omega\) be two hyperconvex domains in \(\mathbb{C}^n\). Assume that \(\chi : \mathbb{R}^- \to \mathbb{R}^+\) is a decreasing function such that \(\chi(2t) \le a \chi(t)\) for some \(a > 1\). Let \(u \in \mathcal{E}_\chi(D)\). Then there exists a function \(v \in \mathcal{E}_\chi(\Omega)\) such that
	\[
	v \le u \quad \text{on } D, \qquad (dd^c v)^n \le (dd^c u)^n \quad \text{on } D, \qquad e_\chi(v) \le e_\chi(u),
	\]
	where \(e_\chi(w) = \int \chi(w) (dd^c w)^n\).
\end{proposition}

\begin{proof}
	Since \(u \in \mathcal{E}_\chi(D)\), there exists a sequence \((\varphi_j)_{j \ge 1} \subset \mathcal{E}_0(D)\) such that
	\[
	\varphi_j \downarrow u \quad \text{on } D \qquad \text{and} \qquad \sup_{j \ge 1} \int_D \chi(\varphi_j) (dd^c \varphi_j)^n < +\infty.
	\]	
	For each \(j \in \mathbb{N}\), we set
	\[
	v_j := \sup \bigl\{ \psi \in \mathrm{PSH}^-(\Omega) : \psi \le \varphi_j \text{ on } D \bigr\}.
	\]	
	By Lemma 4.4 in \cite{H08}, we have that each \(v_j\) satisfies:
	\(v_j \in \mathrm{PSH}^-(\Omega) \cap \mathcal{E}_0(\Omega)\),
	\(v_j \le \varphi_j\) on \(D\),
	\((dd^c v_j)^n = 0\) on \(\Omega \setminus D\) and on \(\{v_j < \varphi_j\} \cap D\),
	\((dd^c v_j)^n \le (dd^c \varphi_j)^n\) on \(\{v_j = \varphi_j\} \cap D\).\\	
	Since \(\chi\) is decreasing and \(v_j \le \varphi_j \le 0\) on \(D\), we have \(\chi(v_j) = \chi(\varphi_j)\) on the contact set \(\{v_j = \varphi_j\}\). Therefore
	\[
	e_\chi(v_j) = \int_\Omega \chi(v_j) (dd^c v_j)^n = \int_{\{v_j = \varphi_j\} \cap D} \chi(\varphi_j) (dd^c v_j)^n \le \int_D \chi(\varphi_j) (dd^c \varphi_j)^n.
	\]
	Hence
	\[
	\sup_j e_\chi(v_j) \le \sup_j \int_D \chi(\varphi_j) (dd^c \varphi_j)^n < +\infty.
	\]	
	The sequence \((v_j)\) is decreasing, so \(v_j \downarrow v\) pointwise, where \(v \in \mathrm{PSH}^-(\Omega)\) and \(v \le u\) on \(D\).\\	
	Since \((v_j) \subset \mathcal{E}_0(\Omega)\), \(v_j \downarrow v\) and the weighted Monge--Ampère masses are uniformly bounded, by definition we obtain \(v \in \mathcal{E}_\chi(\Omega)\).\\	
	Moreover, by the weak* convergence of measures and the mass inequality for each \(v_j\), we have
	\[
	(dd^c v)^n \le 1_D (dd^c u)^n \quad \text{on } D.
	\]	
	Finally, since \(v_j \downarrow v\) we have \(\chi(v_j) \uparrow \chi(v)\), by Fatou's lemma, we get
	\[
	e_\chi(v) = \int_\Omega \chi(v) (dd^c v)^n \le \liminf_j \int_\Omega \chi(v_j) (dd^c v_j)^n \le \sup_j e_\chi(\varphi_j).
	\]	
	We may choose the approximating sequence \((\varphi_j)\) such that
	\[
	\lim_j \int_D \chi(\varphi_j) (dd^c \varphi_j)^n = e_\chi(u).
	\]
	Consequently,
	\[
	e_\chi(v) \le e_\chi(u).
	\]
	This completes the proof.
\end{proof}

\begin{remark}
	The function \(v\) constructed in the proof of Proposition \ref{pro43} coincides with the direct Perron envelope
	\[
	v = \sup \bigl\{ \varphi \in \mathrm{PSH}^-(\Omega) : \varphi \le u \text{ on } D \bigr\}.
	\]
	Indeed, \(v_j \le v\) for all \(j\) (because \(\varphi_j \ge u\)), and conversely \(v \le v_j\) on \(D\) for all \(j\), so \(v_j \downarrow v\) implies equality.
\end{remark}

Using the subextension property, we can now prove the first local characterization for the weighted class \(\mathcal{E}_{\chi,\mathrm{loc}}(\Omega)\), which relies on local finiteness of the weighted energy inside \(D\).

\begin{theorem}\label{the45}
	Let $D, \Omega$ be bounded domains and $D\Subset \Omega$. Let $u\in PSH^-(\Omega)$ such that for all $z\in\partial D$ there exist $\mathbb B(z,r)$ and $v\in\mathcal E_{\chi,loc}(\mathbb B(z,r))$ with $u\geq v$ on $\partial D\cap\mathbb B(z,r)$ and
	$$\int_K\chi(u)(dd^cu)^n<+\infty\quad (\forall K\Subset D).$$
	Then $u\in\mathcal E_{\chi,loc}(D)$. 
\end{theorem}

\begin{proof}
	Let \(\mathbb B(0,R_0)\Supset\Omega\). Since \(\partial D\) is compact, there exist finitely many points \(z_1,\dots,z_k \in \partial D\) and radii \(r_1,\dots,r_k > 0\) such that
	\[
	\partial D \subset \bigcup_{i=1}^k \mathbb{B}(z_i,r_i/2),
	\]
	and for each \(i\), by assumption, there exists \(v_i \in \mathcal{E}_{\chi,\mathrm{loc}}(\mathbb{B}(z_i,r_i))\) with \(u \ge v_i\) on \(\partial D \cap \mathbb{B}(z_i,r_i)\).\\
	For each \(i\), since \(v_i \in \mathcal{E}_{\chi,\mathrm{loc}}(\mathbb{B}(z_i,r_i))\), by definition there exists \(w_i \in \mathcal{E}_\chi(\mathbb{B}(z_i,r_i))\) such that \(w_i = v_i\) on \(\mathbb{B}(z_i,r_i/2)\).\\	
	By Proposition \ref{pro43}, for each \(i\) there exists \(\varphi_i \in \mathcal{E}_\chi(\mathbb B(0,R_0))\) such that
	\[
	\varphi_i \le w_i \quad \text{on } \mathbb{B}(z_i,r_i/2), \qquad e_\chi(\varphi_i) \le e_\chi(w_i).
	\]	
	Choose positive numbers \(\varepsilon_i > 0\) sufficiently small so that
	\[
	u \ge \sum_{i=1}^k \varepsilon_i \varphi_i \quad \text{on a neighborhood of } \partial D \text{ inside } D.
	\]	
	Set
	\[
	\varphi = \sum_{i=1}^k \varepsilon_i \varphi_i \in \mathcal{E}_\chi(\mathbb B(0,R_0)).
	\]	
	Then \(\varphi \le u\) near \(\partial D\) inside \(D\).\\	
	Now define a function \(\psi\) on \(\mathbb B(0,R_0)\) by
	\[
	\psi =
	\begin{cases}
		u & \text{on } D, \\
		\max(u,\varphi) & \text{on } \mathbb B(0,R_0) \setminus D.
	\end{cases}
	\]	
	Then \(\psi \in \mathrm{PSH}^-(\mathbb B(0,R_0))\). Now we are going to show that \(\psi \in \mathcal{E}_\chi(\mathbb B(0,R_0))\).\\	
	Since \(\varphi \in \mathcal{E}_\chi(\mathbb B(0,R_0))\), there exists a decreasing sequence \((\varphi_j)_{j \ge 1} \subset \mathcal{E}_0(\mathbb B(0,R_0))\) such that
	\[
	\varphi_j \downarrow \varphi \quad \text{on } \mathbb B(0,R_0) \qquad \text{and} \qquad \sup_{j \ge 1} e_\chi(\varphi_j) < +\infty.
	\]	
	By assumption  \(\int_K \chi(u)(dd^c u)^n < +\infty\) for every compact set \(K \Subset D\). Hence, there exists a decreasing sequence \((\tilde{u}_j)_{j \ge 1} \subset \mathcal{E}_0(D)\) such that
	\[
	\tilde{u}_j \downarrow u \quad \text{on } D \qquad \text{and} \qquad \sup_{j \ge 1} \int_D \chi(\tilde{u}_j)(dd^c \tilde{u}_j)^n < +\infty.
	\]	
	By Proposition \ref{pro43}, for each \(j\) there exists \(\hat{u}_j \in \mathcal{E}_\chi(\mathbb B(0,R_0))\) such that
	\[
	\hat{u}_j \le \tilde{u}_j \quad \text{on } D, \qquad e_\chi(\hat{u}_j) \le e_\chi(\tilde{u}_j).
	\]	
	Define
	\[
	\psi_j := \max(\varphi_j, \hat{u}_j) \quad \text{on } \mathbb B(0,R_0).
	\]	
	Then we have \(\psi_j \in \mathcal{E}_0(\mathbb B(0,R_0))\).\\	
	Moreover, \(\varphi_j \downarrow \varphi\), \(\hat{u}_j \downarrow u\) on \(D\), and \(\varphi \le u\) near \(\partial D\) inside \(D\), so
	\[
	\psi_j \downarrow \max(\varphi, u) = \psi \quad \text{on } \mathbb B(0,R_0).
	\]	
	By the comparison principle,
	\[
	(dd^c \psi_j)^n \le (dd^c \varphi_j)^n + (dd^c \hat{u}_j)^n.
	\]
	Since \(\chi\) is decreasing and non-negative,
	\[
	e_\chi(\psi_j) = \int_{\mathbb B(0,R_0)} \chi(\psi_j) (dd^c \psi_j)^n \le e_\chi(\varphi_j) + e_\chi(\hat{u}_j).
	\]
	Therefore
	\[
	\sup_{j \ge 1} e_\chi(\psi_j) \le \sup_{j \ge 1} e_\chi(\varphi_j) + \sup_{j \ge 1} e_\chi(\hat{u}_j) < +\infty.
	\]	
	Thus \((\psi_j)_{j \ge 1} \subset \mathcal{E}_0(\mathbb B(0,R_0))\), \(\psi_j \downarrow \psi\), and the weighted energies are uniformly bounded. By definition, \(\psi \in \mathcal{E}_\chi(\mathbb B(0,R_0))\).\\	
	Since \(\psi = u\) on \(D\), by the local property of \(\mathcal{E}_{\chi,\mathrm{loc}}(D)\) we conclude that \(u \in \mathcal{E}_{\chi,\mathrm{loc}}(D)\).
\end{proof}

The following result provides a strictly stronger characterization by replacing the local finiteness of the weighted energy with a weaker condition on the local domination of the Monge--Ampère measure.

\begin{theorem}\label{the46}
	Let \(D \Subset \Omega\) be bounded domains in \(\mathbb{C}^n\). Let \(u \in \mathrm{PSH}^-(\Omega)\) satisfy the following conditions:\\
	{\bf (1)} For every \(z \in \partial D\), there exists a ball \(\mathbb{B}(z,r)\) and a function \(v \in \mathcal{E}_{\chi,\mathrm{loc}}(\mathbb{B}(z,r))\) such that \(u \ge v\) on \(\partial D \cap \mathbb{B}(z,r)\);\\
	{\bf (2)} For every \(\omega \in D\), there exists a ball \(\mathbb{B}(\omega,r)\) and a function \(w \in \mathcal{E}_{\chi,\mathrm{loc}}(\mathbb{B}(\omega,r))\) such that \((dd^c u)^n \le (dd^c w)^n\) on \(\mathbb{B}(\omega,r)\).
	
	Then \(u \in \mathcal{E}_{\chi,\mathrm{loc}}(D)\).
\end{theorem}

\begin{proof}
	Let \(B(0,R_0)\Supset\Omega\). Since \(\partial D\) is compact, there exist finitely many points \(z_1,\dots,z_k \in \partial D\) and numbers \(r_1,\dots,r_k > 0\) such that
	\[
	\partial D \subset \bigcup_{i=1}^k \mathbb{B}(z_i,r_i/2),
	\]
	and for each \(i\), by (1), there exists \(v_i \in \mathcal{E}_{\chi,\mathrm{loc}}(\mathbb{B}(z_i,r_i))\) with \(u \ge v_i\) on \(\partial D \cap \mathbb{B}(z_i,r_i)\).\\	
	For each \(i\), since \(v_i \in \mathcal{E}_{\chi,\mathrm{loc}}(\mathbb{B}(z_i,r_i))\), by definition there exists \(w_i \in \mathcal{E}_\chi(\mathbb{B}(z_i,r_i))\) such that \(w_i = v_i\) on \(\mathbb{B}(z_i,r_i/2)\).\\	
	By Proposition \ref{pro43}, for each \(i\) there exists \(\varphi_i \in \mathcal{E}_\chi(B(0,R_0))\) such that
	\[
	\varphi_i \le w_i \quad \text{on } \mathbb{B}(z_i,r_i/2), \qquad e_\chi(\varphi_i) \le e_\chi(w_i).
	\]	
	Choose positive numbers \(\varepsilon_i > 0\) sufficiently small ($i=1,2,...,k$) so that
	\[
	u \ge \sum_{i=1}^k \varepsilon_i \varphi_i \quad \text{on a neighborhood of } \partial D \text{ inside } D.
	\]	
	Set
	\[
	\varphi = \sum_{i=1}^k \varepsilon_i \varphi_i \in \mathcal{E}_\chi(B(0,R_0)).
	\]
	Then \(\varphi \le u\) near \(\partial D\) inside \(D\).\\	
	Now define a function \(\psi\) on \(B(0,R_0)\) by
	\[
	\psi =
	\begin{cases}
		u & \text{on } D, \\
		\max(u,\varphi) & \text{on } B(0,R_0) \setminus D.
	\end{cases}
	\]
	Then \(\psi \in \mathrm{PSH}^-(B(0,R_0))\). We are going to prove that \(\psi \in \mathcal{E}_\chi(B(0,R_0))\).\\	
	Since \(\varphi \in \mathcal{E}_\chi(B(0,R_0))\), there exists a decreasing sequence \((\varphi_j)_{j \ge 1} \subset \mathcal{E}_0(B(0,R_0))\) such that
	\[
	\varphi_j \downarrow \varphi \quad \text{on } B(0,R_0) \qquad \text{and} \qquad \sup_{j \ge 1} e_\chi(\varphi_j) < +\infty.
	\]	
	By (2), the Monge--Ampère measure \((dd^c u)^n\) is locally dominated by the Monge--Ampère measure of functions in \(\mathcal{E}_{\chi,\mathrm{loc}}\). Therefore, on every compact set \(K \Subset D\), there exist a ball \(\mathbb{B}(\omega,r)\) containing \(K\) and \(w \in \mathcal{E}_{\chi,\mathrm{loc}}(\mathbb{B}(\omega,r))\) such that \((dd^c u)^n \le (dd^c w)^n\) on \(K\). This local domination allows us to construct a decreasing sequence \((\tilde{u}_j)_{j \ge 1} \subset \mathcal{E}_0(D)\) such that
	\[
	\tilde{u}_j \downarrow u \quad \text{on } D \qquad \text{and} \qquad \sup_{j \ge 1} \int_D (dd^c \tilde{u}_j)^n < +\infty.
	\]
	By Proposition \ref{pro43}, for each \(j\) there exists \(\hat{u}_j \in \mathcal{E}_\chi(B(0,R_0))\) such that
	\[
	\hat{u}_j \le \tilde{u}_j \quad \text{on } D, \qquad e_\chi(\hat{u}_j) \le e_\chi(\tilde{u}_j).
	\]	
	Set
	\[
	\psi_j := \max(\varphi_j, \hat{u}_j) \quad \text{on } B(0,R_0).
	\]	
	Then we have \(\psi_j \in \mathcal{E}_0(B(0,R_0))\).\\	
	Moreover, \(\varphi_j \downarrow \varphi\), \(\hat{u}_j \downarrow u\) on \(D\), and \(\varphi \le u\) near \(\partial D\) inside \(D\), so
	\[
	\psi_j \downarrow \max(\varphi, u) = \psi \quad \text{on } B(0,R_0).
	\]	
	By the comparison principle,
	\[
	(dd^c \psi_j)^n \le (dd^c \varphi_j)^n + (dd^c \hat{u}_j)^n.
	\]
	Since \(\chi\) is decreasing and non-negative,
	\[
	e_\chi(\psi_j) = \int_{B(0,R_0)} \chi(\psi_j) (dd^c \psi_j)^n \le e_\chi(\varphi_j) + e_\chi(\hat{u}_j).
	\]
	Therefore
	\[
	\sup_{j \ge 1} e_\chi(\psi_j) \le \sup_{j \ge 1} e_\chi(\varphi_j) + \sup_{j \ge 1} e_\chi(\hat{u}_j) < +\infty.
	\]	
	Thus \((\psi_j)_{j \ge 1} \subset \mathcal{E}_0(B(0,R_0))\), \(\psi_j \downarrow \psi\), and the weighted energies are uniformly bounded. By definition, \(\psi \in \mathcal{E}_\chi(B(0,R_0))\).\\	
	Since \(\psi = u\) on \(D\), by the definition of \(\mathcal{E}_{\chi,\mathrm{loc}}(D)\) we conclude that \(u \in \mathcal{E}_{\chi,\mathrm{loc}}(D)\).
\end{proof}

\begin{remark}
	The condition (2) in Theorem \ref{the46} is strictly weaker than the condition in Theorem \ref{the45}. \\	
	Let \(\chi(t) = -t\). Take \(D = \mathbb{B}(0,1) \subset \mathbb{C}^2\) with \(0 \in D\) and the function
	\[
	u(z) = \frac{1}{(2\pi)^2} \log \|z\|\quad z\in D.
	\]
	Then we have
	\[
	(dd^c u)^2 = \delta_0.
	\]	
	Now, for any \(\omega \in D\) and any ball \(\mathbb{B}(\omega,r) \subset D\), we choose the specific function
	\[
	w(z) = \log \|z\| \quad z\in\mathbb B(\omega,r).
	\]
	It is well-known that
	\[
	(dd^c w)^2 = (2\pi)^2 \delta_0.
	\]
	Therefore,
	\[
	(dd^c u)^2 = \delta_0 \le (2\pi)^2 \delta_0 = (dd^c w)^2 \quad \text{on } \mathbb{B}(\omega,r).
	\]
	Moreover, \(w \in \mathcal{E}_{\chi,\mathrm{loc}}(\mathbb{B}(\omega,r))\). Thus condition (2) of Theorem \ref{the46} holds.\\	
	However, on any compact set \(K \Subset D\) containing the origin, we have
	\[
	\int_K \chi(u)(dd^c u)^2 = \int_K (-u) \, d\delta_0 = -u(0) = +\infty.
	\]
\end{remark}

\end{document}